\documentclass[12pt]{amsart}
\usepackage{geometry}
\usepackage{graphicx}
\usepackage{amssymb}
\usepackage{epstopdf}
\usepackage{amsmath,amscd}
\usepackage{amsthm}
\usepackage{url,verbatim}

\RequirePackage[colorlinks,citecolor=blue,urlcolor=blue]{hyperref}
\usepackage{breakurl}
\theoremstyle{plain}
\newtheorem{theorem}{Theorem}
\newtheorem{definition}[theorem]{Definition}
\newtheorem{lemma}[theorem]{Lemma}
\newtheorem{proposition}[theorem]{Proposition}

\newtheorem{example}[theorem]{Example}

\newtheorem{remark}[theorem]{Remark}

\newcommand\es{\varnothing}
\newcommand\wt{\widetilde}
\newcommand\ol{\overline}
\newcommand\Aut{\mathrm{Aut}}

\newcommand\sH{{\mathcal H}}
\newcommand\RR{{\mathbb R}}
\newcommand\ZZ{{\mathbb Z}}
\newcommand\NN{{\mathbb N}}

\renewcommand\a{\alpha}

\newcommand\g{\gamma}

\newcommand\be{\beta}
\newcommand\si{\sigma}
\newcommand\eps{\epsilon}
\renewcommand\th{\theta}
\newcommand\De{\Delta}
\newcommand\qq{\qquad}
\newcommand\q{\quad}
\newcommand\resp{respectively}
\newcommand\spz{\mathrm{span}}

\newcommand\oo{\infty}
\newcommand\sG{{\mathcal G}}

\newcommand\sC{{\mathcal C}}

\newcommand\Ga{\Gamma}
\newcommand\La{\Lambda}
\newcommand\Si{\Sigma}
\newcommand\de{\delta}
\newcommand\id{{\bf 1}}

\newcommand\EFG{\text{\rm EFG}}
\newcommand\phie{{\phi^\eta}}

\renewcommand\o{\text{\rm o}}

\newcommand\ghf{graph height function}
\newcommand\ughf{unimodular graph height function}

\newcommand\hdi{$\sH$-difference-invariant}
\newcommand\normal{\trianglelefteq}

\newcommand\sghf{strong \ghf}
\newcommand\phimin{{\phi_{\text{\rm inf}}}}
\newcommand\phimax{{\phi_{\text{\rm sup}}}}
\newcommand\thmin{\th_{\text{\rm min}}}
\newcommand\thmax{\th_{\text{\rm max}}}
\newcommand\supp{{\text{\rm supp}}}

\newcommand\dsup{d_{\text{\rm sup}}}
\newcommand\Dsup{D_{\text{\rm sup}}}

\newcommand\Phifin{\Phi_{\text{\rm bnd}}}

\newcommand\len{\ell}
\newcommand\lenmin{\len_{\text{\rm inf}}}
\newcommand\lenmax{\len_{\text{\rm sup}}}

\newcounter{mycount}

\newenvironment{numlist}{\begin{list}{\arabic{mycount}.}%
   {\usecounter{mycount}\labelwidth=1cm\itemsep 0pt}}{\end{list}}
\newenvironment{letlist}{\begin{list}{\rm(\alph{mycount})}%
   {\usecounter{mycount}\labelwidth=1cm\itemsep 0pt}}{\end{list}}

\numberwithin{equation}{section}
\numberwithin{theorem}{section}
\numberwithin{figure}{section}

\title{Weighted self-avoiding walks}
\author{Geoffrey R.\ Grimmett}
\address{Statistical Laboratory, Centre for
Mathematical Sciences, Cambridge University, Wilberforce Road,
Cambridge CB3 0WB, UK}
\address{School of Mathematics \&\ Statistics, The University of Melbourne, 
Parkville, VIC 3010, Australia}
\email{g.r.grimmett@statslab.cam.ac.uk}
\urladdr{\url{http://www.statslab.cam.ac.uk/~grg/}}

\author{Zhongyang Li}
\address{Department of Mathematics,
University of Connecticut,
Storrs, Connecticut 06269-3009, USA} 
\email{zhongyang.li@uconn.edu}
\urladdr{\url{http://www.math.uconn.edu/~zhongyang/}}

\begin{document}

\begin{abstract}
We study the connective constants of weighted self-avoiding walks (SAWs) on infinite graphs and groups.
The main focus is upon weighted SAWs on finitely generated, virtually indicable groups. 
Such groups possess so-called \lq height functions', and this permits the study of SAWs 
with the special property of being bridges.  The group structure is relevant
in the interaction between the height function and the weight function.
The main difficulties arise when the support of the weight function is unbounded, since
the corresponding graph is no longer locally finite. 

There are two principal results, of which the first is a condition under which the 
weighted connective constant and the weighted bridge constant are equal.
When the weight function has unbounded support, we work
with a generalized notion of the \lq length' of a walk, which is subject to a certain condition.   

In the second main result, the above equality is used to 
prove a continuity theorem for connective constants
on the space of weight functions endowed with a suitable distance function.
\end{abstract}

\date{April 15, 2018, accepted 5 June 2019}

\keywords{Self-avoiding walk, weight function, connective constant, bridge constant, 
virtually indicable group, Cayley graph, transitive graph}

\subjclass[2010]{05C30, 20F65, 82B20}

\maketitle

\section{Introduction}\label{sec:intro}

The \emph{counting} of self-avoiding walks (SAWs) is extended here to the study of 
\emph{weighted} SAWs. Each edge is assigned a weight, and the weight of a SAW 
is defined as the product of the weights of its edges. 
We study certain properties of the exponential growth rate $\mu$ (known
as the \emph{connective constant}) in terms
of the weight function $\phi$, including its continuity on
the space of weight functions. 

The theory largely follows the now
established route when the underlying graph $G$ is locally finite. 
New problems emerge when $G$ is not locally finite, and these are explored
in the situation in which $G$ is the complete graph on an infinite group $\Ga$, with
weight function $\phi$ defined on $\Ga$. 

One of the main technical steps in the current work is the proof, subject to 
certain conditions, of the equality of 
the weighted connective constant and the weighted bridge constant. This 
was proved in \cite[Thm 4.3]{GL-loc} for the unweighted 
constants on any connected, infinite,
quasi-transitive, locally finite, simple graph possessing a unimodular 
\ghf.\footnote{Since
the current paper was written, Lindorfer \cite{Lind19} has obtained a corresponding
statement without the assumption of unimodularity.} 
This result is extended here to weighted SAWs on finitely generated, 
virtually indicable groups 
(see Section \ref{sec:gps} and \cite{Hill94,Hill09} for references on virtual indicability).

As explained in \cite{GrLrev2,GL-loc}, a key assumption for the definition and study
of \emph{bridge} SAWs on an unweighted graph 
is the existence of a so-called \lq\ghf' $h$.
A \lq bridge' is a SAW $\pi$ for which there exists an interval $[a,b]$ such that
the vertices of $\pi$ have heights lying in $[a,b]$, and the initial (\resp, final)
vertex of $\pi$ has height $a$ (\resp, $b$). The \lq bridge constant' $\be$
is the exponential growth rate of the number of bridges with length $n$
and fixed initial vertex.
We note that some Cayley graphs (including those of virtually indicable groups)
possess \ghf s, and some do not (see \cite{GL-Cayley}).

Let $\phi$ be a non-negative weight function on the edge-set of a graph, and 
assume $\phi$ is invariant under a
certain class of graph automorphisms. The weight of a SAW $\pi$ is defined to
be the product of the weights of the edges of $\pi$. 
By replacing the \emph{number}
of SAWs by their \emph{aggregate weight} in the above, 
we may define the connective constant
and the bridge constant of
this weighted system of SAWs. The question arises of whether these two numbers
are equal
for a given graph $G$, \ghf\ $(h,\sH)$, and weight function $\phi$.
Our proof of the equality of bridge and connective constants,
given in Sections \ref{sec:gps} and \ref{sec:pf3.6},  hinges on a
combinatorial fact (due to Hardy and Ramanujan, \cite{HR},
see Remark \ref{rem3}) whose application
here imposes a condition on the pair $(\phi,h)$. This condition generally  fails
in the non-locally-finite case, but holds if the usual combinatorial 
definition of the length of a SAW (that is, the number
of its edges) is replaced by a generalized length function $\len$ satisfying
a certain property $\Pi(\len,h)$, stated in \eqref{eq:cond1}. 
In the locally finite case,
the usual graph-distance function invariably satisfies $\Pi(\len,h)$.
 
Certain new problems arise when working with a general length function $\len$
other than the usual graph-distance,
but the reward includes the above desired equality, and also the continuity of the weighted 
connective constant on the space
of weight functions with a suitable distance function. 

Here is an informal summary of our principal results,
presented in the context of weighted walks on finitely generated groups.
The reader is referred to 
Sections \ref{sec:sawb}--\ref{sec:cty} (and in particular, Theorems \ref{thm:group} and 
\ref{thm:cont2}) for formal definitions and statements.  

\begin{theorem}\label{thm:0}
Let $\Ga$ be an infinite,  finitely generated, virtually indicable group,
and let $\phi:\Ga\to[0,\oo)$ be a summable, symmetric weight function that spans $\Ga$.
If \eqref{eq:cond1} holds, then
\begin{letlist}
\item  the connective and bridge constants are equal,
\item
the connective constant is a continuous function
on the space $\Phi$ of such weight functions endowed with a suitable distance function.
\end{letlist}
\end{theorem}

Similar results are given for locally finite quasi-transitive graphs, 
see Theorems \ref{thm1} and \ref{thm3}.

The principal results of the paper concern weighted SAWs on certain groups including,
for example,  all finitely generated,
elementary amenable groups. The defining property of the groups under study is that they 
possess so-called \lq\sghf s' 
(see \cite{GL-loc, GL-amen, GrLrev2} and especially \cite{GL-Cayley});
this is shown in Theorem \ref{thm:indic} to be essentially equivalent to
assuming the groups to be virtually indicable.
The algebraic structure
of a group $\Ga$ plays its role through the interaction between the 
height function $h$, the length function $\len$,  and the weight function $\phi$. 

When the support of the weight function $\phi$ is bounded, 
the results of this paper
are fairly straightforward extensions of the unweighted case (see Section \ref{sec:sawb}).
The situation is significantly more complicated 
when $\phi$ has unbounded support,
as in Section \ref{sec:gps}. Single steps of a SAW may have unbounded length
in the usual graph metric, and thus this constitutes a \lq long-range' model,
in the language of statistical mechanics.  Long-range models
have been studied in numerous contexts including
percolation and the Ising model (see, for example,
\cite{CMPR, DingS} and the references therein). We are unaware of   
applications of our
SAW results to other long-range models, and indeed the principal complication
in the current work, namely the restriction to length-functions satisfying 
\eqref{eq:cond1}, does not appear to have a parallel in other systems.

Weighted SAWs have featured in earlier work of others in 
a variety of contexts.
\emph{Randomly weighted} SAWs on grids and trees
have been studied by Lacoin \cite{Lac1,Lac2} and Chino and Sakai
\cite{Chino, ChS}. The work of Lacoin is 
directed at counting SAWs inside the infinite cluster of
a supercritical percolation process, so that the effective
weight function takes the values
$0$ and $1$.
A type of weighted SAW on the square grid $\ZZ^2$
has been considered by Glazman and Manolescu \cite{Glaz,GlazM},
who show that the ensuing  two-point  function is independent of the weights so long
as they conform to the corresponding Yang--Baxter equation.
This may be viewed as 
a partial extension of results of Duminil-Copin and Smirnov \cite{ds} 
concerning SAWs on the trangular lattice. 

Here is a summary of the contents of the article.
The asymptotics of weighted self-avoiding walks on 
quasi-transitive, locally finite graphs are
considered in Section \ref{sec:sawb}. The relevance of
\ghf s is explained in the context of bridges, and the equality of
connective and  bridge constants is proved for graphs possessing \ughf s.
In this case, the connective constant is continuous on the space of weight 
functions with the supremum norm. Proofs are either short or even omitted, since
only limited novelty is required beyond \cite{GL-loc}.

Weighted walks on a class of infinite countable groups are the subject of Sections 
\ref{sec:gps} and \ref{sec:cty}, namely, on the class of  virtually indicable groups. 
A generalized notion of the  length
of a SAW is introduced, and a condition is established under which the bridge constant
equals the connective constant. The connective constant 
is shown to be continuous in the weight and length functions. 
Proofs are largely deferred
to Sections \ref{sec:pf1}--\ref{sec:pfcont}.

We write $\RR$ for the reals, $\ZZ$ for the integers, and $\NN$ for the natural numbers.

\section{Weighted walks and bridges on locally finite graphs}\label{sec:sawb}

In this section we consider weighted SAWs on locally finite graphs.
The \emph{length} of a walk is its conventional graph length, that is, the number of its edges.

Let $\sG$ be the set of connected, infinite,
quasi-transitive, locally finite, simple, rooted graphs, with root labelled $\id$.
Let $G=(V,E)\in\sG$.
We weight the edges of $G$ via a function  $\phi:E\to(0,\oo)$.
For $\sH\le \Aut(G)$, the function $\phi$ is called $\sH$-\emph{invariant} if
$$
 \phi\bigl(\langle\a u,\a v\rangle\bigr)= \phi\bigl(\langle u,v\rangle\bigr), \qq 
 \langle u,v\rangle \in E,\ \a\in\sH.
$$
We write
$$
\phimin=\inf\{\phi(e): e \in E\}, \qq \phimax=\sup\{\phi(e): e \in E\}.
$$
If $\phi$ is $\sH$-invariant and $\sH$ acts quasi-transitively, then 
\begin{equation}\label{eq:minmax}
0<\phimin\le \phimax <\oo.
\end{equation}

A walk on $G$ is called \emph{$n$-step} if it traverses exactly $n$ edges
(possibly with reversals and repeats).
An \emph{$n$-step self-avoiding walk} (SAW) on $G$ is an ordered sequence 
$\pi=(\pi_0,\pi_1,\dots,\pi_n)$
of distinct vertices of $G$ such that $e_i=\langle \pi_{i-1},\pi_i\rangle \in E$ 
for $1\le i\le n$. The \emph{weight} of $\pi$ is the product
\begin{equation}\label{eq:wtprod}
w(\pi) := \prod_{i=1}^n \phi\bigl(\langle \pi_{i-1},\pi_i\rangle \bigr).
\end{equation}
Note that the weight function $w=w_\phi$ acts \emph{symmetrically}
in that the weight of an edge or SAW is the same irrespective of the direction
in which it is traversed. 
Let $\Si_n(v)$ be the set of $n$-step SAWs starting at $v\in V$, 
and set $\Si_n=\Si_n(\id)$. For a set $\Pi$ of SAWs, we write
$$
w(\Pi)= \sum_{\pi\in \Pi}w(\pi)
$$
for the total weight of members of $\Pi$.

The following is proved  as for the unweighted case in
\cite{jmhII} (see also \cite{hm,Lac2,ms}), and the proof is omitted.
For $\sH\le\Aut(G)$, we let $\Phi_E(\sH)$ be the space of functions $\phi:E\to(0,\oo)$
that are $\sH$-invariant. 

\begin{theorem}\label{thm:concon}
Let $G=(V,E)\in \sG$, and let $\sH\le\Aut(G)$  act quasi-transitively on $G$.
Let $\phi\in\Phi_E(\sH)$.
There exists $\mu=\mu(G,\phi)\in(0,\oo)$, called the \emph{connective
constant}, such that
$$
\lim_{n\to\oo} w(\Si_n)^{1/n} = \mu.
$$
\end{theorem}

Let $G=(V,E)\in\sG$. It was shown in \cite{GL-loc,GL-Cayley} (see also the review \cite{GrLrev2})
how to define \lq bridge' SAWs on certain general families of graphs, 
and how to adapt the proof of Hammersley and Welsh
\cite{HW62} to show equality of the connective constant and the bridge constant. 
Key to this approach is the following notion of a \ghf. See \cite{GL-Cayley,LyP}
for accounts of unimodularity.

\begin{definition} \label{def:height}
A \emph{\ghf} on $G=(V,E)\in\sG$ is a pair $(h,\sH)$ such that:
\begin{letlist}
\item $h:V \to\ZZ$, and $h(\id)=0$, 
\item $\sH\le\Aut(G)$ acts quasi-transitively on $G$, and 
$h$ is \emph{\hdi} in that
$$
h(\a v) - h(\a u) = h(v) - h(u), \qq \a \in \sH,\ u,v \in V,
$$
\item for  $v\in V$,
there exist neighbours $u$, $w$ such that
$h(u) < h(v) < h(w)$.
\end{letlist}
A \ghf\ $(h,\sH)$  is called \emph{unimodular} if the action of
$\sH$ on $V$ is unimodular.
\end{definition}

Associated with a \ghf\ $(h,\sH)$
are two integers $d$, $r$ which we define next. Let
\begin{equation}\label{eq:defd}
d=d(h)=\max\bigl\{|h(u)-h(v)|: u,v\in V,\ u \sim v\bigr\},
\end{equation}
where $u \sim v$ means that $u$ and $v$ are neighbours.

If $\sH$ acts transitively, we set
$r=0$. Assume $\sH$ does not act transitively, and 
let $r=r(h,\sH)$ be the infimum of all $r$ such that the following holds.
Let $o_1,o_2,\dots,o_M$ be representatives of the orbits of $\sH$.
For $i\ne j$, there
exists $v_j \in \sH o_j$ 
such that $h(o_i)<h(v_j)$, and a SAW $\nu_{i,j}:=\nu(o_i,v_j)$ from $o_i$
to $v_j$, with length $r$ or less, all of whose vertices $x$, other than
its endvertices,  satisfy $h(o_i)<h(x)< h(v_j)$.
We fix such SAWs $\nu_{i,j}$, and we set $\nu_{i,i}=\{o_i\}$.
These SAWs will be used in Sections \ref{sec:pf1} and \ref{sec:pf3.6}. 
Meanwhile, set
\begin{equation}\label{eq:phimin-}
\phi_\nu =\min\bigl\{\phi(e): e \in \nu_{i,j} \text{ for some } i,j\bigr\},
\end{equation}
and
\begin{equation}\label{eq:phim}
\th_{i,j} = w(\nu_{i,j}), \qq \thmin=\min_{i,j} \th_{i,j},\qq
\thmax=\max_{i,j} \th_{i,j},
\end{equation}
where $\th_{i,i}:=1$.
If $\sH$ acts transitively, we set $\phi_\nu = \thmin=\thmax=1$.
Some properties of $r$ and $d$ have been established in  \cite[Prop.\ 3.2]{GL-loc}.

We turn to so-called half-space walks and bridges.
Let $G=(V,E)\in \sG$ have a \ghf\ $(h,\sH)$, and let $\phi\in\Phi_E(\sH)$.
Let $v \in V$ and $\pi=(\pi_0,\pi_1,\dots,\pi_n)\in\Si_n(v)$.
We call $\pi$ a \emph{half-space SAW} if 
$$
h(\pi_0)<h(\pi_i), \qq 1\leq i\leq n,
$$
and we write $H_n(v)$ for the set of half-space walks $\pi$
with initial vertex $v$.
We call $\pi$ a \emph{bridge} if
\begin{equation}\label{eq:bridge}
h(\pi_0)<h(\pi_i) \leq h(\pi_n), \qq 1\leq i\leq n,
\end{equation}
and a \emph{reversed bridge} if \eqref{eq:bridge} is replaced by
$$
h(\pi_n)\le h(\pi_i) < h(\pi_0), \qq 1\leq i\leq n.
$$
 
The \emph{span} of a SAW $\pi$ is defined as 
$$
\spz(\pi) = \max_{0\leq i\leq n}h(\pi_i)-\min_{0\leq i\leq n}h(\pi_i).
$$
Let $\be_n(v)$ be the set of $n$-step bridges
$\pi$ from $v$, and let
\begin{equation}\label{eq:defbinf}
w(\be_n):= \min\{w(\be_n(v)): v \in V\}.
\end{equation}
It is easily seen (as in \cite{HW62}) that
\begin{equation}\label{eq:b-subadditive}
w(\be_{m+n}) \ge w(\be_m)w( \be_n),
\end{equation}
from which we deduce the existence of the \emph{bridge constant}
\begin{equation}\label{eq:bexists}
\be  = \be(G,h,\sH) = \lim_{n\to\oo} w(\be_n)^{1/n}
\end{equation}
satisfying
\begin{equation}\label{eq:bless}
w(\be_n) \le \be^n, \qq n \ge 0.
\end{equation}

\begin{remark}\label{rem:hdep}
The bridge constant $\be$ depends on the choice
of \ghf\ $(h,\sH)$ and weight function $\phi\in\Phi_E(\sH)$. 
We shall see in Theorem \ref{thm1}
that its value is constant (and equal to $\mu$)
across the set of \emph{unimodular} \ghf s $(h,\sH)$
for which $\phi$ is $\sH$-invariant. 
\end{remark}

\begin{proposition}\label{prop:mequalsb}
Let  $G=(V,E)\in\sG$ have a \ghf\ $(h,\sH)$, and suppose 
$\phi\in\Phi_E(\sH)$. Then
\begin{equation}\label{eq:bvconv}
w(\be_n(v))^{1/n} \to \be, \qq v \in V,
\end{equation}
and furthermore 
\begin{equation}\label{eq:bnabove}
w(\be_n(v)) \le \be^n(\be/\phi_\nu)^r, \qq n\ge 1,\ v \in V,
\end{equation}
where $r=r(h,\sH)$ and $\phi_\nu$ are given after \eqref{eq:defd}.
\end{proposition}

\begin{theorem}[Weighted bridge theorem]\label{thm1}
Let $G=(V,E)\in\sG$ possess a \ughf\ $(h,\sH)$,
and let $\phi\in\Phi_E(\sH)$. Then $\be = \mu$.
\end{theorem}

The proof of Theorem \ref{thm1}  is summarized in Section \ref{sec:pf2}. 
It is essentially that of \cite[Thm 1]{GL-loc}. The supremum metric of the next theorem 
will appear in a more general form in \eqref{eq:dist3}.

\begin{remark}\label{lind1}
{\rm(Added on final revision.)} Let $G \in \sG$ have \ghf\ $(h,\sH)$,
and consider unweighted SAWs on $G$.
Since the current paper was written, Lindorfer \cite{Lind19} has
found an expression for the connective constant $\mu=\mu(G)$ 
without assuming unimodularity, namely
$\mu=\max\{\be^+,\be^-\}$
where $\be^\pm$ is the bridge constant associated with the height function 
$(\pm h,\sH)$. We have by \cite{GL-loc} that $\be^+=\be^-$
if $(h,\sH)$ is unimodular. Furthermore, unimodularity is necessary for $\be^+=\be^-$,
as exemplified by the natural \ghf\ on Trofimov's
grandparent graph \cite{trof85}.
Corresponding statements are valid in the weighted case also.
\end{remark}

\begin{theorem}\label{thm3}
Let $G=(V,E)\in\sG$ have a \ghf\ $(h,\sH)$.
We endow $\Phi_E(\sH)$ with the supremum metric 
\begin{equation}\label{eq:new902}
\dsup(\phi,\psi) =\sup\bigl\{|\phi(e)-\psi(e)|: e\in E\bigr\}.
\end{equation}
The constants $\mu$ and $\be$ are continuous functions on $\Phi_E(\sH)$.
More precisely,
\begin{equation}\label{eq:prec}
\left(1-\frac{\dsup(\psi,\phi)}\phimin\right)\le
\frac{\mu(G,\psi)}{\mu(G,\phi)},\frac{\be(G,\psi)}{\be(G,\phi)} \le 
 \left(1+\frac{\dsup(\psi,\phi)}\phimin\right), \q \phi,\psi \in \Phi_E(\sH).
\end{equation}
\end{theorem}

\begin{proof}[Proof of Proposition \ref{prop:mequalsb}]
Assume $G$ has \ghf\  $(h, \sH)$. If $G$ is transitive, the claim is trivial
by \eqref{eq:bexists} and \eqref{eq:bless}, 
so we assume $G$ is quasi-transitive but not transitive.

Choose $x\in V$ such that $w(\be_{n+r}(x))=w(\be_{n+r})$. 
Let $v \in V$, and let $v$ have type $o_j$ and $x$ type $o_i$.
Let $\nu(o_i,v_j)$ be given as above \eqref{eq:phimin-}. 
The length $l(o_i,v_j)$ of $\nu(o_i,v_j)$
satisfies $l(o_i,v_j)\le r$. Find $\eta\in\sH$ such that $\eta(v_j)=v$, and let $\nu(x,v)=\eta(\nu(o_i,v_j))$.
Let $l=l(o_i,v_j)$ if $i\ne j$,
and $l=0$ otherwise. Then,
$$
\phi_\nu^{r}  w(\be_n(v)) \le \phi_\nu^{r-l} w(\be_{n+l}(x)) \le w(\be_{n+r}(x)) =w(\be_{n+r}),
$$
and \eqref{eq:bnabove} follows by \eqref{eq:bless}.
The limit \eqref{eq:bvconv} follows by 
\eqref{eq:defbinf} and \eqref{eq:bexists}.
\end{proof}

\begin{proof}[Proof of Theorem \ref{thm3}]
We shall show continuity at the point $\phi\in\Phi_E(\sH)$.
For $\pi\in\Si_n$,
\begin{align*}
\left(1-\frac{\dsup(\psi,\phi)}\phimin\right)^n\le
\frac{w_\psi(\pi)}{w_\phi(\pi)} = \prod_{e\in \pi} \frac{\psi(e)}{\phi(e)}
\le \left(1+\frac{\dsup(\psi,\phi)}\phimin\right)^n,
\end{align*}
from which \eqref{eq:prec} follows.
\end{proof}

\section{Weight functions with unbounded support}\label{sec:gps}

We consider weighted SAWs on groups in this section. If the weight function has bounded support, 
the relevant graph is locally finite, and the methods and conclusions of the last section apply.
New methods are needed if the support is unbounded, and it turns out to be useful to consider 
a different measure of the length of a walk. We shall consider groups that support \ghf s.

A group $\Ga$ is
called \emph{indicable} if there exists a surjective homomorphism $F: \Ga\ \to\ZZ$.
It is called \emph{virtually indicable}
if there exists $\sH \normal \Ga$ with $[\Ga:\sH]<\oo$
such that $\sH$ is indicable. If this holds, we call $\Ga$ virtually $\sH$-indicable (or,
sometimes, virtually $(\sH,F)$-indicable). (See, for example, \cite{Hill94,Hill09}
for information on virtual indicability.)

We remind the reader of the definition of a \emph{strong} \ghf, as derived from \cite[Defn 3.4]{GL-amen}.

\begin{definition}\label{def:sghf}
A (not necessarily locally finite) Cayley graph $G=(\Ga,F)$ of a finitely generated group $\Ga$ is said to have a
\emph{\sghf} $(h,\sH)$ if the following two conditions hold.
\begin{letlist}
\item $\sH\normal \Ga$ acts on $\Ga$ by left-multiplication, and $[\Ga:\sH]<\oo$.
\item $(h,\sH)$ is a \ghf\ on $G$.
\end{letlist}
\end{definition}

The group properties of: (i) virtual indicability, and (ii) the possession of a \sghf, are
equivalent in the sense of the next theorem.

\begin{theorem}%[{\cite[Thm 4.2]{GL-amen}, \cite[Thm 3.4]{GL-Cayley}}]
\label{thm:indic}
Let $\Ga$ be a finitely generated group.
\begin{letlist}
\item If $\Ga$ is virtually $(\sH,F)$-indicable,
any (not necessarily locally finite) Cayley graph $G=(\Ga,F)$ of 
$\Ga$ possesses a \sghf\ of the form $(h,\sH)$.
\item If some Cayley graph $G$ of $\Ga$ possesses a \sghf, denoted $(h,\sH)$,
then $\Ga$ is virtually $\sH$-indicable.
\end{letlist}
\end{theorem}

In the context of this theorem, $\sH$ acts freely on the Cayley graph $G$,
and is therefore unimodular. Hence,
the functions $(h,\sH)$ are \ughf s on $G$. 
The proofs of this and other results in this section may be found in Section \ref{sec:pf1}.

Here is some notation.
Let $\Ga$ be finitely generated and virtually $(\sH,F)$-indicable.
It is shown in \cite[Sect.\ 7]{GL-Cayley} that, for a given locally finite Cayley graph $G$,  
there exists a unique harmonic extension $\psi$ of $F$, 
and that $\psi$ takes rational values. It is then shown how to construct a harmonic, \sghf\ 
for $G$ of the form
$(h,\sH)$. The latter construction is not unique. For given $(\Ga, \sH, F, G)$, we write $h=h_F$ 
for a given such function.

By Theorems \ref{thm1} and \ref{thm:indic}, the weighted bridge constant equals the 
weighted connective
constant for any \emph{locally finite}, weighted Cayley graph of a 
finitely generated, virtually indicable group. Here is an example of a class of such groups.

\begin{example}[Elementary amenability]
Let $\EFG$ denote the class of infinite, finitely generated, elementary amenable groups.
It is standard that any  $\Ga\in\EFG$ is virtually indicable.
See, for example, \cite[Proof of Thm 4.1]{GL-Cayley}.
\end{example}

We turn to weight functions and SAWs.
Let $\Ga$ be virtually $\sH$-indicable, and let $\phi:\Ga\to [0,\oo)$ satisfy
$\phi(\id)=0$ and
be \emph{symmetric} in that
$$
\phi(\g^{-1})=\phi(\g), \qq \g\in \Ga.
$$
The \emph{support} of $\phi$ is the set $\supp(\phi):=\{\g: \phi(\g)>0\}$.
We call $\phi$ \emph{summable} if
\begin{equation}\label{eq:205}
w(\Ga):=\sum_{\g\in\Ga} \phi(\g)
\end{equation}
satisfies $w(\Ga)<\oo$. We write
$$
\phimax=\sup\{\phi(\g): \g\in\Ga\}.
$$
When $\phi$ is summable, we have that
\begin{equation}\label{eq:finite}
\phimax<\oo, \qq \big|\{\gamma\in \Ga:\phi(\gamma)\geq C\}\big| < \oo \text{ for } C > 0.
\end{equation}
We misuse notation by setting $\phi(\a,\g) = \phi(\a^{-1}\g)$ for $\a,\g\in\Ga$.

We consider SAWs on the complete graph $K=(\Ga,E)$ with
vertex-set  $\Ga$ and weights $w=w_\phi$
as in \eqref{eq:wtprod}. We say that $\phi$ \emph{spans} $\Ga$ if,
for $\eta,\g\in\Ga$, there exists a SAW on $K$ from $\eta$ to $\g$
with strictly positive weight. This is equivalent to requiring that $\supp(\phi)$ generates $\Ga$. 
Let $\Phi$ be the set of functions $\phi:\Ga\to [0,\oo)$ that are 
symmetric, summable, and which span $\Ga$, and let $\Phifin$ be the subset 
of $\Phi$ containing functions $\phi$ with bounded support. Functions in
$\Phi$ are called \emph{weight functions}.

Let $\phi\in\Phi$.
Since we shall be interested only in SAWs $\pi$ with strictly positive weights $w(\pi)$,
we shall have use for the subgraph 
\begin{equation}\label{eq:Gphi}
K_\phi=(\Ga,E_\phi) \q\text{where}\q 
E_\phi =\bigl\{\langle u,v\rangle: u^{-1}v \in \supp(\phi)\bigr\}.
\end{equation}

\begin{proposition}\label{t42}
Let $\Ga$ be finitely generated and virtually $\sH$-indicable, and let $\phi\in\Phi$. 
The graph $K_\phi$ of \eqref{eq:Gphi}  
possesses a \sghf\ of the form $(h,\sH)$.
\end{proposition}

\begin{proof}
This is a corollary of Theorem \ref{thm:indic}(a), since $\supp(\phi)$ is a generator-set
of $\Ga$.
\end{proof}

%Since $\sH$ is assumed indicable, there exists a surjection
%$F:\sH\to\ZZ$ such that
%\begin{letlist}
%\item $F(\id)=0$,
%\item $F$ is non-constant,
%\item $F(\g\eta) = F(\g) + F(\eta)$ for $\g,\eta\in\sH$.
%\end{letlist}

We turn to the notion of the \emph{length} of a walk. 
Let $\La$ be the space of all functions $\len:\Ga\to (0,\oo)$ satisfying
\begin{letlist}
\item $\len$ is symmetric in that 
$\len(\g)=\len(\g^{-1})$ for $\g\in\Ga$,
\item $\len$  satisfies
\begin{equation}\label{eq:lenmin}
\lenmin>0 \q\text{where}\q \lenmin:=\inf\{\len(\g): \g\in\Ga\}.
\end{equation}
\end{letlist}
We extend the domain of $\len\in\La$ to the edge-set $E$ by
defining the \emph{$\len$-length} of $e \in E$ to be 
$$
\len(e) =\len(u^{-1}v), \qq e=\langle u, v\rangle \in E. 
$$
The $\len$-length of a SAW $\pi=(\pi_0,\pi_1,\dots, \pi_n)$ is thus
$$
\len(\pi) = \sum_{i=1}^{n} \len(e_i), \qq e_i=\langle \pi_{i-1},\pi_i\rangle.
$$
Evidently, 
\begin{equation}
 n\leq \len(\pi)/\lenmin.
 \label{lwl}
\end{equation}

\begin{example}\label{ex:comb}
Let $\phi\in\Phi$ and suppose $\len\equiv 1$. Then $\len$ restricted
to $K_\phi$ is the usual graph-distance.
\end{example}

\begin{example}\label{ex:weighted}
Let $\phi\in\Phi$, and suppose $\len(\g)=1/\phi(\g)$ for $\g\in\supp(\phi)$.
The values taken by $\len$ off $\supp(\phi)$ are immaterial to the study
of weighted SAWs $\pi$, since  $w_\phi(\pi)=0$ if $\pi$ traverses any edge
not in $\supp(\phi)$.
\end{example}

Let $\Ga$ be virtually $(\sH,F)$-indicable, and 
let $(h_F,\sH)$ be the \sghf\ constructed after Theorem \ref{thm:indic}.
A SAW $\pi$ on $K$ 
is a \emph{bridge} (with respect to $h_F$) if \eqref{eq:bridge} holds.
Let $\phi\in\Phi$ and $\len\in\La$. For $v \in V$, 
let $\Si(v)$ (\resp, $B(v)$) be the set of SAWs (\resp, bridges) on $G$
starting at $v$, and abbreviate $\Si(\id)=\Si$. For $c>0$, let
\begin{equation}
\begin{aligned}
\sigma_{m,c}^{\len}(v)&=\{\pi\in \Si(v):m \le \len(\pi) < m+c\},\\
\be_{m,c}^{\len}(v)&=\{\pi\in B(v):m \le \len(\pi) < m+c\}.
\end{aligned}\label{mp-}
\end{equation}
We shall abbreviate $\si_{m,c}^{\len}(v)$ (\resp, $\be_{m,c}^{\len}(v)$)
to $\si_{m,c}(v)$ (\resp, $\be_{m,c}(v)$) when confusion is unlikely to arise. 
Since $K$ is transitive and $\len$ is $\Ga$-invariant, 
$\si_{m,c}(v)$ does not depend on the choice of $v$, and we write
$\si_{m,c}$ for its common value. In contrast, $\be_{m,c}(v)$ may depend
on $v$ since the height function $h$ is generally only $\sH$-difference-invariant. 
Although we shall take limits as $m\to\oo$
through the integers, it will
be useful sometimes to allow $m$ to be non-integral in \eqref{mp-}.

\begin{theorem}\label{prop:1}
Let $\Ga$ be finitely generated and virtually $(\sH,F)$-indicable with \sghf\ $h_F$, and let $\phi\in\Phi$
and $\len\in\La$.
\begin{letlist}
\item
For all sufficiently large $c$, the limits
\begin{equation}\label{mp}
\mu_{\phi,\len}=\limsup_{m\rightarrow\infty}w_\phi (\sigma_{m,c})^{{1}/{m}},\q
\beta_{\phi,\len}=\lim_{m\rightarrow\infty} w_\phi(\beta_{m,c}(v))^{{1}/{m}},\qq v \in \Ga,
\end{equation}
exist and are independent of the choice of $c$ (and of the choice of $v$, in the latter case). They
are called the (weighted) \emph{connective constant} and (weighted) 
\emph{bridge constant}, \resp.
\item 
If $\lenmax:=\sup\{\len(\g): \g\in\Ga,\, \g\in\supp(\phi)\}$ satisfies $\lenmax<\oo$,
then $\limsup$ may be replaced by $\lim$ in \eqref{mp}.
\item
We have that 
\begin{equation}\label{eq:207}
0 < \be_{\phi,\len}\le \mu_{\phi,\len} \le\max\bigl\{1,w(\Ga)^{1/\lenmin}\bigr\}.
\end{equation}
\end{letlist}
\end{theorem}

\begin{example}[A non-summable weight function]
 It is possible to have $w(\Ga)=\oo$ and $\mu_{\phi,\len}<\oo$. 
Let $\Ga=\ZZ$, and let $K=(\ZZ,E)$ be the complete graph on $\Ga$. Take $\len$ 
and $\phi$ by
\begin{equation*}
\len(n)=|n|, \q \phi(n) = \frac 1{|n|}, \qq n\in \ZZ,
\end{equation*}
whence $w(\Ga)=\infty$. The edge $\langle m,n\rangle\in E$
has  length $\len(\langle m,n\rangle)=|m-n|$. 
Let $G_0=(\ZZ,E_0)$, where $E_0$ is the set of unordered pairs $m,n$ with $|m-n|=1$. 
Let $\pi=(n_0=0,n_1,n_2,\dots,n_k)$ be a SAW on $K$ with
$\len$-length $\len(\pi)=m$. 
For $1\leq i\leq k$, let $\nu_i$ be the shortest path on $G_0$ from $n_{i-1}$ to $n_i$,
and let $\wt\pi$ be the $m$-step path on $G_0$ obtained by following the paths $\nu_1,\nu_2,\dots,\nu_k$ in sequence. 
There are $2^m$ $m$-step walks on $G_0$ from $0$, and each such walk $\wt\pi$
arises (as above) from at most $2^m$ distinct SAWs $\pi$ on $K$. 
Since each SAW $\pi$ on $K$ has weight satisfying $w(\pi)\le 1$, we have 
$w(\si_{m,1})\le 4^m$, so that $\mu_{\phi,\len}\leq 4$.
\end{example}

We have no proof in general of the full convergence 
$w_\phi(\si_{m,c})^{1/m} \to \mu_{\phi,\len}$, 
but we shall see in Theorem \ref{thm:group} that this holds subject to an additional condition. 

Theorem \ref{prop:1}(a) assumes a lower bound on the 
length $c$ of the intervals in \eqref{mp-}.
That some lower bound is necessary is seen by considering Example \ref{ex:weighted}
with a weight function $\phi$  every non-zero value 
of which has the form $2^{-s}$ for some integer $s\ge 1$. When this holds, every
$\len$-length is a multiple of $2$, implying that $\si_{m,c}, \be_{m,c} = \es$ 
if $m$ is odd and $c<1$.
The required lower bounds on $c$ are discussed after the statement 
of Proposition \ref{prop:beta}; see \eqref{eq:102}.

We introduce next a condition on the triple $(\phi,\len,h_F)$.
Let $\Ga$ be virtually $(\sH,F)$-indicable with \sghf\ $(h_F,\sH)$,
and let $\phi\in\Phi$.
For $\eps\in [1, 2)$ and $C>0$, let $\La_\eps=\La_\eps(C,\phi)$ be the
subset of $\La$ containing functions $\len$ satisfying the following H\"older condition
for the height function:
\begin{equation}\label{eq:cond1}
|h_F(u)-h_F(v)|\leq C[\len(u^{-1}v)]^\eps, \qq  u,v\in\Ga,\ \phi(u^{-1}v)>0.
\end{equation}
Note that the usual graph-distance of Example \ref{ex:weighted}
satisfies \eqref{eq:cond1} (with $\eps=1$ and suitable $C$)
if $\phi$ has bounded support; the converse is false, as illustrated in Example \ref{ex5}.

\begin{example}\label{gen-dist}
Let $\Ga$ be finitely generated with finite generating set $S$ and \ghf\ $(h,\sH)$,
and let $G_S$ be the corresponding Cayley graph of $\Ga$.
For $u\in \Ga$, let $\len(u)$ be the number of edges in the shortest
path of $G_S$ from $\id$ to $u$. Then $\ell$ satisfies \eqref{eq:cond1}
with $\eps=1$ and suitable $C$. Readers may prefer to fix their ideas on such $\len$.
Property \eqref{eq:cond1} will feature briefly in the proofs of the ensuing theorems.
\end{example}

Here is our first main theorem, proved in Section \ref{sec:pf3.6}.
Once again, we require $c \ge A$ where $A$ will be given in \eqref{eq:102}.

\begin{theorem}\label{thm:group}
Let $\Ga$ be  finitely generated and virtually $(\sH,F)$-indicable with \sghf\ $(h_F,\sH)$, 
and let $\phi\in\Phi$ and $\len\in\La_\eps(C,\phi)$ for some $\eps\in [1,2)$ and $C>0$.
\begin{letlist}
\item
We have that $\mu_{\phi,\len}=\be_{\phi,\len}$, and furthermore,
for all sufficiently large $c$,
\begin{equation}\label{eq:sigmalim}
\mu_{\phi,\len}=\lim_{m\to\oo} w_\phi (\sigma_{m,c})^{{1}/{m}}.
\end{equation}
\item
{\rm (Monotonicity of $\mu$)}
If $\nu\in\Phi$ and $\nu\le\phi$, then $\len\in\La_\eps(C,\nu)$ and
$\mu_{\nu,\len} \le \mu_{\phi,\len}$.
\end{letlist}
\end{theorem}

A sufficient condition for \eqref{eq:cond1} is presented in Lemma \ref{l44}.

\begin{remark}[Condition \eqref{eq:cond1}]\label{rem3}
Some condition of type \eqref{eq:cond1} is necessary 
for the proof that $\mu_{\phi,\len}=\be_{\phi,\len}$
(presented in Section \ref{sec:pf3.6}) for the following reason. 
A key estimate in the proof of \cite{HW62} concerning the bridge constant 
is the classical Hardy--Ramanujan \cite{HR} estimate of $\exp\bigl(\pi\sqrt{n/3}\bigr)$  
for the number of ordered partitions of an integer $n$.
It is important for the proof that this number is $e^{\o(n)}$. Inequality \eqref{eq:cond1} implies
that the aggregate height difference along a SAW $\pi$ has order no greater than 
$n:= \lfloor (\len(\pi))^\eps\rfloor$ with
$\eps<2$, which has $\exp\bigl(\o(\len(\pi))\bigr)$ ordered partitions.
(See Proposition \ref{prop2}.)
\end{remark}

\begin{remark}[Working with graph-distance]\label{rem4}
Let us work with the graph-distance of Example \ref{ex:comb}
(so that the length of a walk is the number of its edges), 
with $\Si_n$ and $\be_n$ given as usual.
Then $w(\Si_n)$ is sub-multiplicative, whence the limit $\mu=\lim_{n\to\oo} w(\Si_n)^{1/n}$ exists.
By Theorem \ref{prop:1}(c), $\mu<\oo$ if $w(\Ga)<\oo$.
It is not hard to see the converse, as follows. Suppose $w(\Ga)<\oo$.
Let
$\pi=(\pi_0=\id,\pi_1,\dots,\pi_n)$ be a SAW with $w(\pi)>0$, 
and write $\Ga'=\Ga\setminus\{\pi_j^{-1}: 0 \le j \le n\}$.
Then $\Si_{n+1}$ contains the walks $\{\gamma+\pi: \gamma\in\Ga'\}$, where 
$\gamma+\pi:=(\id,\gamma,\gamma \pi_1, \dots, \gamma \pi_n)$. Therefore, $w(\Si_{n+1}) \ge w(\Ga')w(\pi)=\oo$
for $n \ge 1$.

We have no general proof
of the equality of $\mu$ and the bridge limit $\be$ when $\supp(\phi)$ is unbounded,
although this holds 
subject to \eqref{eq:cond1} with $\len$ taken as the graph-distance of Example
\ref{ex:comb}.
\end{remark}

\begin{example}[Still working with graph-distance]\label{ex5}
Let $\phi\in\Phi$ and  take $\len$ to be graph-distance on $K_\phi$, as in Example \ref{ex:comb}.
By Theorem \ref{prop:1}(b), the limit $\mu_{\phi,\len}=\lim_{m\to\oo} w(\si_{m,c})^{1/m}$
exists, and it is easy to see that one may take any $c>0$. 
Here is an example in which $K_\phi$ is not locally finite.
Let $\Ga=\ZZ^2$, 
and let
\begin{equation*}
\phi(\langle p,q\rangle)=\begin{cases}
\bigl(|p|^2+|q|^2\bigr)^{-1}&\text{if either $p=0$,  or $(p,q)=(\pm 1, 0)$},\\ 
0&\text{otherwise}.
\end{cases}
\end{equation*}
Then $K_\phi$ is the Cayley graph of $\Gamma$ in which $(m,n)$ and $(m',n')$ 
are joined by an edge if and only if either $m=m'$ or $(m',n')=(m\pm 1, n)$.
Define the height function by $h(m,n)=m$, and note that \eqref{eq:cond1} holds
for suitable $C$.
\end{example}

\begin{remark}\label{rem1}
Let $\Ga$ be finitely generated and virtually $(\sH,F)$-indicable, 
and let $\phi\in\Phi$ and $\len\in\La_\eps(C,\phi)$ for some $\eps\in[1,2)$ and $C>0$. 
Since $\mu_{\phi,\len}$ is independent of the
choice of triple $(h_F,\sH,F)$, so is $\be_{\phi,\len}$.
\end{remark}

\section{Continuity of the connective constant}\label{sec:cty}

A continuity theorem for connective constants is proved in this section.
We work on the space $\Phi\times\La$ with distance function
\begin{align}\label{eq:dist3}
\Dsup\bigl((\phi,\len_1), (\psi,\len_2)\bigr)
= \sup_\g \left|\frac{\phi(\g)-\psi(\g)}{\phi(\g)+\psi(\g)}\right|+ \sup_\g \left| \frac{\len_1(\g)-\len_2(\g)}{\len_1(\g)+\len_2(\g)}\right|,
%\nonumber
\end{align}
where the  suprema are  over the  set $\supp(\phi)\cup\supp(\psi)$.
This may be compared with \eqref{eq:new902}.

The following continuity theorem is our second main theorem, and it is proved in Section \ref{sec:pfcont}. 
For $C,W>0$ and $\eps\in[1,2)$, we write $\Phi\circ\La_\eps(C,W)$
for the space of all pairs $(\phi,\len)$ satisfying $\phi\in\Phi$, 
$\len \in \La_\eps(C,\phi)$, and $\mu_{\phi,\len}\le W$.

\begin{theorem}[Continuity theorem]\label{thm:cont2}
Let $\Ga$ be finitely generated and virtually $(\sH,F)$-indicable. 
Let  $C,W>0$ and $\eps\in[1,2)$. The connective constant $\mu$
is a continuous function on the space $\Phi\circ \La_\eps(C,W)$ 
endowed with the distance function $\Dsup$ of \eqref{eq:dist3}.
\end{theorem}

Section \ref{sec:pfcont} contains also a proposition concerning the effect 
on a connective constant $\mu_{\phi,\len}$ of truncating
the weight function $\phi$; see Proposition \ref{prop:trunc}. 

\section{Proofs of Theorems \ref{thm:indic} and \ref{prop:1}}\label{sec:pf1}

We prove next the aforesaid theorems together with some
results in preparation for the proof in Section \ref{sec:pf3.6} of Theorem 
\ref{thm:group}. We suppress explicit reference to $\phi$ and $\len$ except
where necessary to avoid ambiguity.

\begin{lemma}\label{l41}
Let $\Ga$ be finitely generated and let $G$ be a 
(not necessarily locally finite) Cayley graph of $\Ga$ with generator-set
$S=(s_i: i \in I)$.
There exists a locally finite Cayley graph $G'$ of $\Ga$ 
with (finite) generator-set $S'\subseteq S$.
\end{lemma}

\begin{proof}
Let $S_0$ be a  finite generator-set of $\Ga$. 
Each $s\in S_0$ can be expressed as a  finite product of the form
$s=\prod \{\psi: \psi\in T_s\}$ with $T_s \subseteq S$. 
We set $S' = \bigcup_{s\in S_0}T_s$.
\end{proof}

\begin{proof}[Proof of Theorem \ref{thm:indic}]
(a)  
Let $G$ be a Cayley graph of $\Ga$ with generator-set $S$.
By Lemma \ref{l41}, we can construct a locally finite Cayley graph 
$G'$ of $\Ga$ with 
respect to a finite set  $S'\subseteq S$ of generators.  
Since $\sH$ acts freely and is therefore unimodular,
we may apply \cite[Thm 3.4]{GL-Cayley}
to deduce that $G'$ has a \sghf\ of the form $(h,\sH)$. 
Since $G'$ is a subgraph of $G$ with the same vertex set, and $h$ is $\sH$-difference-invariant on $G'$,
it is $\sH$-difference-invariant on $G$ also.

It remains to check that each vertex $v\in \Ga$ has neighbours $u$, $w$ in $G$
with $h(u)<h(v)<h(w)$. 
This holds since there exist $s_1,s_2\in S'\subseteq S$ such that 
$h(vs_1)<h(v)<h(vs_2)$.

(b)
Let $G$ be a Cayley graph of $\Ga$ with a \sghf\ $(h,\sH)$. Then $\sH \normal \Ga$ and
$[\Ga:\sH]<\oo$. Let $\lambda=\gcd\{|h(u)|: u \in \sH,\ h(u) \ne 0\}$.
Then $\Ga$ is virtually $(\sH,F)$-indicable with $F =h/\lambda$.
\end{proof}

There follows a sufficient condition for the condition \eqref{eq:cond1} of the principal
Theorem \ref{thm:group}.

\begin{lemma}\label{l44}
Let $\Ga$ be finitely generated and $\phi\in\Phi$. Let
$G$ be the Cayley graph with generator-set $\supp(\phi)$, 
and let $G'$ be as in Lemma \ref{l41}. Write
$\de'$ for graph-distance on $G'$. Let $(h_F,\sH)$ be the \sghf\ described afterTheorem \ref{thm:indic}.
If there exists $C_1>0$ and $\eps\in[1,2)$ such that
\begin{equation}\label{eq:cond2}
\len(v) \ge  C_1 n^{1/\eps}, \qq v\in \Ga,\ \de'(\id,v)=n \ge 1,
\end{equation}
then there exists $C>0$ such that \eqref{eq:cond1} holds.
\end{lemma}

\begin{proof}
Since $(h_F,\sH)$ is a \ghf\ on the locally finite graph $G'$,  by \eqref{eq:defd}
there exists $d\in(0,\oo)$ such that
\begin{equation*}
|h_F(u)-h_F(v)| \le dn, \qq u,v\in\Ga,\ \de'(u,v)=n.
\end{equation*}
Inequality \eqref{eq:cond1} follows by \eqref{eq:cond2} with $C=dC_1^{-\eps}$.
\end{proof}

\begin{lemma}\label{l46}
Let $\Gamma$ be finitely generated and virtually $\sH$-indicable, where $\sH\ne \Ga$, 
and let $\phi\in\Phi$. 
Let $K_\phi=(\Gamma,E_\phi)$ be as in \eqref{eq:Gphi}, and let $(h,\sH)$ be a \sghf\
on $K_\phi$. Let $o_i$ be a representative of the $i$th orbit of $\sH$.
There exists an integer  $s=s(h,\sH,\phi)\in \NN$ such that the following holds.
For $i \ne j$, there exists $v_j\in\sH o_j$ such that $h(o_i)<h(v_j)$, and 
a SAW $\nu_{i,j}=\nu(o_i,v_j)$ of $G'$
from $o_i$ to $v_j$ with $\len$-length $s$ or less, each of whose vertices $x$ 
other than its endvertices satisfy $h(o_i)<h(x)<h(v_j)$.
\end{lemma}

\begin{proof}
This extends to $K_\phi$ the discussion above \eqref{eq:phimin-}.
Let $G'$ be a locally finite Cayley graph of $\Gamma$ with respect to a finite 
generator-set 
$S\subseteq \supp(\phi)$, as in Lemma \ref{l41}. By \cite[Prop.\ 3.2]{GL-loc},
there exists $0<r<\infty$ such that, for $i\ne j$, there exists $v_j\in\sH o_j$ such that
$h(o_i)<h(v_j)$, and a SAW $\nu_{i,j}=\nu(o_i,v_j)$ on $G'$ from $o_i$ to $v_j$ with 
graph-length $r$ or less. Furthermore, each of the vertices $x$ of $\nu(o_i,v_j)$
other than its endvertices satisfy $h(o_i)<h(x)<h(v_j)$.

Since $G'$ is a subgraph of $G$, a SAW on $G'$ is also a SAW on $G$. 
Moreover, since $\sH$ acts quasi-transitively,  
and every edge of $G'$ has finite $\len$-length, 
we can find $s\in (0,\infty)$ such that the claim holds.
\end{proof}

\begin{lemma}\label{prop:beta}
Let $\Ga$ be finitely generated and virtually $(\sH,F)$-indicable
with \sghf\ $h_F$, 
and let $\phi\in\Phi$ and $\len\in\La$. 
There exist $\be_{\phi,\len}\in(0,\oo)$ and $A=A(\phi,\sH, F)$ such that, for $c \ge A$,
\begin{equation}\label{eq:betaAlim}
\lim_{m\to\oo} w(\be_{m,c}(v))^{1/m} = \be_{\phi,\len}, \qq v \in \Ga.
\end{equation}
The constant $A$ is given in the forthcoming equations \eqref{eq:102}.
\end{lemma}

In advance of the proof, we introduce some notation that will be useful later in this work.
Let $G'=(\Ga, E')$ be the  Cayley graph generated by the finite $S' \subseteq \supp(\phi)$
of Lemma \ref{l41} applied to $G:=K_\phi$. 
By Theorem \ref{thm:indic}(a) and its proof, there exists $(h_F,\sH)$
which is a \sghf\ of both $K_\phi$ and $G'$.

Since $h_F$ is a \ghf\ on $G'$,
for $v \in \Ga$, there exists an edge $e_v=\langle v,v\a_v\rangle\in E'$ 
such that $h_F(v\a_v)>h_F(v)$ 
(and also $\phi(\a_v)>0$ by definition of $G'$). We call the edge $e_v$
 the \emph{extension} at $v$.
%\begin{equation}\label{eq:206}
%\a_v= \argmax\{\phi(\g): h_F(v\g)>h(v)\}.
%\end{equation}
Write
\begin{equation}\label{eq:102}
\begin{gathered}
\psi=\inf\{\phi(\a_v): v \in \Ga\},\\
a=\inf\{ \len(\a_v):v \in \Ga\},
\q  A=\sup\{\lceil \len(\a_v)\rceil:v \in \Ga\},
\end{gathered}
\end{equation}
noting that $\psi,a>0$ and $A<\oo$ since 
$h_F$ is \hdi\ and $\sH$ acts quasi-transitively.  
Let
\begin{equation}\label{eq:105}
w(\be_{m,A})=\inf\{w(\be_{m,A}(v)): v \in \Ga\}.
\end{equation}

\begin{lemma}\label{lem:inter}
For $q \ge c-A\ge 0$, there exists $C_1=C_1(\psi,q,c,a)\in (0,\oo)$ such that 
$$
w(\be_{p,c}(v)) \le C_1 w(\be_{p+q,A}(v)),\qq p \in (0,\oo),\ v\in \Ga.   
$$
\end{lemma}

\begin{proof}
Let $\pi \in \be_{p,c}(v)$. 
We may extend $\pi$ by adding progressive extensions at the final endpoint.
After some number $r$ of such extensions, we achieve a bridge $\pi'$
contributing to  $\be_{p+q,A}$. 
Note that $w(\pi') \ge \psi^{r}w(\pi)$, and $r \le \lceil q/a \rceil$. 
Each such $\pi'$ occurs no more than $\lceil c/a\rceil$ times in this construction.
The claimed inequality holds with 
$C_1=\lceil c/a\rceil \psi^{-\lceil q/a\rceil}$.
\end{proof}

\begin{proof}[Proof of Lemma \ref{prop:beta}]
By concatenation as usual,
$$
w(\be_{m,A})w(\be_{n.A}) \le w(\be_{m+n,2A}(v)), \qq v \in \Ga.
$$
By Lemma \ref{lem:inter}, there exists $C=C(\psi,A,a)$ such that
$$
w(\be_{m,A})w(\be_{n,A}) \le C w(\be_{m+n+A,A}(v)).
$$
This holds for all $v$, whence
\begin{equation}\label{eq:101}
w(\be_{m,A})w(\be_{n,A}) \le C w(\be_{m+n+A,A}).
\end{equation}
We deduce (by subadditivity, as in \cite[eqn (8.38)]{G99}, for example) the existence of the limit 
$$
\be_{\phi,\len}=\be_{\phi,\len,A} :=\lim_{m\to\oo} w(\be_{m,A})^{1/m}.
$$
In particular, by \eqref{eq:105},
\begin{equation}\label{eq:110}
\liminf_{m\to\oo} w(\be_{m,A}(v))^{1/m} \ge \be_{\phi,\len},\qq v \in \Ga.
\end{equation}
By \eqref{eq:101},
\begin{equation}\label{eq:209}
w(\be_{m,A}) \le C^{-1}(\be_{\phi,\len})^{m+A} .
\end{equation}

We show next that \eqref{eq:betaAlim} (with $c=A$) is valid for all $v \in \Ga$.
Let $s$ be as in Lemma \ref{l46}, and 
let $x,v\in\Ga$ have types $o_i$, $o_j$, \resp, where $i \ne j$. 
Any $b_v\in \be_{m,A}(v)$ may be prolonged 
`backwards' by a translate $\a\nu(o_i,v_j)$, for suitable
$\a\in\sH$,  to obtain a bridge
$b_x $ from $x$ with weight $w(b_x) = \th_{i,j}w(b_v)$,
where $\th_{i,j}= w(\nu(o_i,v_j))$ as in \eqref{eq:phim}. By Lemma \ref{lem:inter},
there exists $C_1=C_1(\psi,A,a,s)$ such that
$$
\th_{i,j}w(\be_{m,A}(v))  \le C_1 w(\be_{m+s,A}(x).
$$
If $i=j$, this holds with $\th_{i,j}=1$.

We pick $x$ such that $w(\be_{m+s,A}(x))=w(\be_{m+s,A})$.
By \eqref{eq:105} and \eqref{eq:209}, 
equation \eqref{eq:betaAlim} follows for $v$, with $c=A$, and in addition
\begin{equation}\label{eq:208}
w(\be_{m,A}(v))\le C_2 (\be_{\phi,\len})^{m+s+A},\qq v \in \Ga,
\end{equation}
by \eqref{eq:209}, where $C_2=C_2(\psi,A, a,\thmin,s)$.

Equation \eqref{eq:betaAlim} with $c=A$ follows by \eqref{eq:110} and \eqref{eq:208}. 
When $c > A$, it suffices to note that
\begin{equation}\label{eq:216}
\be_{m,A}(v) \subseteq \be_{m,c}(v) \subseteq \bigcup\bigl\{\be_{m+i,A}(v): 0\le i \le \lceil c-A\rceil\bigr\},
\end{equation}
implying that the value of the limit $\be_{\phi,\len}$ is independent of $c>A$.
By \eqref{eq:102}, $w(\be_{0,A})>0$, and hence $\be_{\phi,\len}>0$ 
by \eqref{eq:208}.
\end{proof}

\begin{proof}[Proof of Theorem \ref{prop:1}]

(a) The second  limit in \eqref{mp} is included in Lemma \ref{prop:beta}.
That $\mu_{\phi,\len}$ is independent of the value of $c\ge A$ holds by \eqref{eq:216} with
$\si$ substituted for $\be$. 

(b)
Let $L=\lceil\lenmax\rceil < \oo$. 
The usual submultiplicative argument for SAWs yields 
\begin{align}\label{eq:567}
w(\si_{m+n,L}) &\le w(\si_{m,L})w( \si_{n-L,2L})\\
&= w(\si_{m,L})\bigl[w(\si_{n-L,L}) + w(\si_{n,L})\bigr].\nonumber
\end{align}
By \eqref{eq:567} with $m=k-L\ge 0$ and $n=L$,
$$
w(\si_{k,L}) \le w(\si_{k-L,L})\bigl[w(\si_{0,L}) + w(\si_{L,L})\bigr],
$$
which we substitute into \eqref{eq:567} with $k=n$ to obtain
$$
w(\si_{m+n,L}) \le C w(\si_{m,L})w(\si_{n-L,L}),
$$
with $C=w(\si_{0,L}) + w(\si_{L,L})$. The claim follows in the usual way with $c=L$.
As at \eqref{eq:209},
\begin{equation}\label{eq:new922}
w(\si_{m,L}) \ge C^{-1}(\mu_{\phi,\len})^{m+L}.
\end{equation}

(c) 
It is trivial that $\be_{\phi,\len} \le \mu_{\phi,\len}$, and it was noted
at the end  of the proof of part (a) that $\be_{\phi,\len}>0$.
Let $\Si_n$ be the set of $n$-step SAWs from $\id$. 
By \eqref{lwl},
\begin{equation*}
w(\sigma_{m,c})\leq \sum_{i=1}^{\lfloor(m+c)/\lenmin\rfloor}w(\Si_i).
\end{equation*}
Equation \eqref{eq:207} follows since $w(\Si_i) \leq w(\Ga)^{i}$.
\end{proof}

\begin{proposition}\label{prop2}
Assume $\phi\in\Phi$ and $\len\in\La_\eps(C,\phi)$ for $\eps \in [1,2)$
and $C>0$. 
For $u,v\in \Ga$ and any SAW $\pi$ on $K_\phi$ from $u$ to $v$,
\begin{equation*}
|h_F(u)-h_F(v)|\leq C \len(\pi)^\eps.
\end{equation*}
\end{proposition}

\begin{proof}
Suppose $\pi$ has vertices $u_i := u\g_1\g_2\cdots \g_i$ for $0\le i < n$.
Now,
\begin{align*}
|h_F(u)-h_F(v)| &\le \sum_{i=1}^n |h_F(u_{i+1})-h_F(u_i)|\\
&\le C\sum_{i=1}^n \len(\g_i)^\eps \qq\text{by \eqref{eq:cond1}}\\
&\le C\left(\sum_{i=1}^n \len(\g_i)\right)^\eps
= C \len(\pi)^\eps,
\end{align*}
as required.
\end{proof}

\section{Proof of Theorem \ref{thm:group}}\label{sec:pf3.6}

Let $\phi\in\Phi$ and $\len\in\La_\eps(C,\phi)$ with $\eps\in [1,2)$ and $C>0$.
With the height function $h=h_F$ given as after Theorem \ref{thm:indic}, and $A$ as in \eqref{eq:102},
we abbreviate $\be_{m,A}^{\len} = \be_m$ and $\si_{m,A}^{\len} =\si_m$.
Clearly, $\be_m \leq \sigma_m$, 
whence $\beta_{\phi,\len}\leq \mu_{\phi,\len}$. 
It suffices, therefore, to show that $\beta_{\phi,\len}\geq \mu_{\phi,\len}$.
We do this by adapting the proof of \cite[Thm 4.3]{GL-loc}; it may also be done
using the group structure of $\Ga$ directly, rather than by referring to unimodularity.
Let $H_m(v)$ be the set of half-space walks from $v$ with $\len$-lengths in the
interval $[m,m+A)$.
Theorem \ref{thm:group} will follow immediately from the following Proposition
\ref{swb4}.

We shall introduce several constants, denoted $D$, which depend on certain parameters as specified. 
The single character $D$ is used repeatedly for economy of notation; 
the value of such $D$ can vary between appearances. We recall the constants
$\thmin$, $\thmax$ of \eqref{eq:phim}, $r$ as after \eqref{eq:defd}, $s$
of Lemma \ref{l46}, and $\psi$, $A$, $a$ of \eqref{eq:102}.

\begin{proposition}\label{hsb3}
There exist constants $D=D(\psi,A,a,\thmin,\thmax,r,s)$, which are continuous functions
of $\psi$, $A$, $a$, $\thmin$, $\thmax$,  such that
%s $C_3=C_3(\psi,\thmin,s)$, $C=C(\psi)$,  and $D=D(\psi,\thmin,r)$ such that 
$$
w(H_m(v)) \le D m^{2\eps}e^{D m^{\eps/2}}(\be_{\phi,\len})^{m+Nm^{\eps/2}}, \qq m \ge 1,
$$
where 
\begin{equation}\label{eq:210}
N= \begin{cases} D &\text{if } \be_{\phi,\len} > 1,\\
0 &\text{if } \be_{\phi,\len} \le 1.
\end{cases}
\end{equation}
\end{proposition}

\begin{proposition}\label{swb4}
There exist constants  $D=D(\psi,A,a,\thmin,\thmax,r,s)$, which are continuous functions
of $\psi$, $A$, $a$, $\thmin$, $\thmax$, such that
$$
w(\si_m)\leq Dm^{4\eps} e^{D m^{\eps/2}}(\be_{\phi,\len})^{m+Nm^{\eps/2}}, \qq m \ge 1,
$$
where $N$ is given in \eqref{eq:210}.
\end{proposition}

\begin{proof}[Proof of Proposition \ref{hsb3}]
Let $\pi=(\pi_0=v,\pi_1,\ldots,\pi_n)\in H_m(v)$ 
have graph-length $n$ and $\len$-length $\len(\pi)$.  
Let $n_0=0$, and for $j \ge 1$,  define $S_j=S_j(\pi)$ and $n_j=n_j(\pi)$ recursively as follows:
\begin{equation*}
S_j=\max_{n_{j-1}\leq t\leq n}(-1)^j\bigl[h(\pi_{n_{j-1}})-h(\pi_t)\bigr],
\end{equation*}
and $n_j$ is the largest value of $t$ at which the maximum is attained. 
The recursion is stopped at the smallest integer $k=k(\pi)$ such that $n_k=n$, 
so that $S_{k+1}$ and $n_{k+1}$ are undefined. 
Note that $S_1$ is the span of $\pi$ and, more generally,
$S_{j+1}$ is the span of the SAW $\ol\pi^{j+1} := (\pi_{n_j},\pi_{n_j+1},\dots,\pi_{n_{j+1}})$. 
Moreover, each of the subwalks $\ol\pi^{j+1}$ is either a bridge or
a reversed bridge. 
We observe that $S_1>S_2>\dots>S_k>0$. 

For a decreasing sequence of $k\ge 1$ positive integers $a_1>a_2>\dots>a_k>0$, 
let $B_{m,c}^v(a_1,a_2,\dots,a_k)$ be the  set of  half-space walks 
$\pi$ from $v\in \Ga$ with $\len$-length 
satisfying $\len(\pi)\in [m, m+c)$, and such that 
$k(\pi)=k$, $S_1(\pi)=a_1$, 
$\dots$, $S_k(\pi)=a_k$, and $n_k(\pi)=n$ 
(and hence $S_{k+1}$ is undefined). We abbreviate $B_{m,A}^v = B_m^v$. 
In particular, $B_m^v(p)$ is the
set of bridges $\pi$
from $v$ with span $p$ and $\len$-length $\len(\pi)\in[m,m+A)$.  

Recall the weight $\th_{i,j}$ of the SAW $\nu_{i,j}=\nu(o_i,v_j)$
as after \eqref{eq:defd}. Let $s_{i,j}$ be the $\len$-length of $\nu_{i,j}$, and
\begin{equation}\label{eq:201}
\de_{i,j}=h(v_j)-h(o_i), \qq  s =\lceil\max\{s_{i,j}\}\rceil,
\end{equation}
as in Lemma \ref{l46}.
We shall perform surgery on $\pi$ to obtain a SAW $\pi'$ satisfying
\begin{align}\label{pi'in11}
\pi' &\in 
\begin{cases} B_{m+2s_{i,j}}^v(a_1+a_2+a_3+2\de_{i,j},a_4,\dots,a_k) &\text{if } k \ge 3,\\
B_{m+s_{i,j}}^v(a_1+a_2+\de_{i,j}) &\text{if } k=2,
\end{cases}\\
w(\pi')&=\begin{cases}
\th_{i,j}^2w(\pi) &\text{if } k\ge 3,\\
\th_{i,j}w(\pi) &\text{if } k =2,
\end{cases}
\label{pi'in22}
\end{align}
for some $i$, $j$ depending on $\pi$.
Note that the subscripts in \eqref{pi'in11} may be non-integral.

\begin{figure}[htbp]
\centerline{\includegraphics[width=0.8\textwidth]{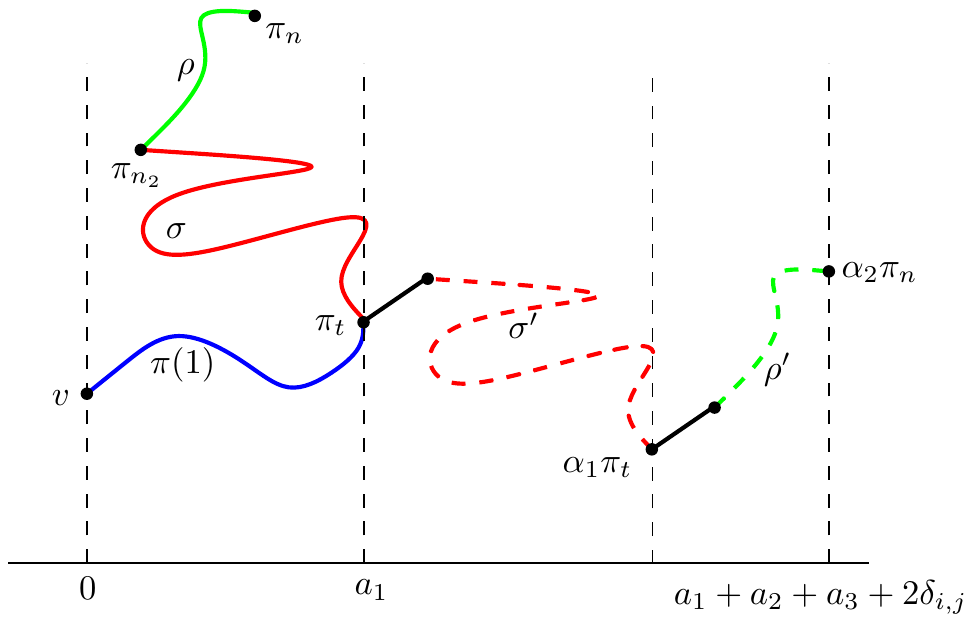}}
     \caption{Two images of $\nu_{i,j}$ are introduced   
     in order to make the required connections. These 
   extra connections are drawn as black straight-line segments, and each has weight
   $\th_{i,j}$, $\len$-length $s_{i,j}$, and span $\de_{i,j}$.} 
   \label{fig:surgery}
\end{figure}

The new SAW $\pi'$ is constructed in the following way. 
Suppose first that $k \ge 3$. In the following, we use the fact that $\sH$ acts 
on $\Ga$ by left-multiplication.
\begin{numlist}
\item Let $t=\min\{u: h(\pi_u)=a_1\}$, and let
$\pi(1)$ be the sub-SAW from $\pi_0=\id$ to the vertex $\pi_t$.
\item Let 
$\si:=(\pi_t,\dots,\pi_{n_2})$ and $\rho=(\pi_{n_2}, \dots, \pi_n)$ be the two sub-SAWs of $\pi$
with the given endvertices.  The type of $\pi_t$ (\resp, $\pi_{n_2}$)
is denoted $o_i$ (\resp, $o_j$). In particular, $\pi_t=\g o_i $ for some $\g\in\sH$.
We map the SAW $\nu_{i,j}=\nu(o_i,v_j)$, given after \eqref{eq:defd}, under $\g$ to
obtain a SAW denoted $\nu(\pi_t,\a_1\pi_{n_2})$, where 
$\a_1=\g v_j\pi_{n_2}^{-1}$ is 
the unique element of $\sH$ such that 
$\a_1 \pi_{n_2} = \g v_j  $. 
Note that $\nu(\pi_t,\a_1\pi_{n_2})$ is
 the single point $\{\pi_t\}$ if $i=j$,
and $\a_1 \pi_{n_2} =\pi_t$ in this case.

The union of the three SAWs $\pi(1)$, $\nu(\pi_t,\a_1\pi_{n_2})$, and 
$\si':=\a_1\si$ (reversed) is a SAW, denoted $\pi(2)$, from $v$ to $\a_1\pi_t$. 
Note that $ h(\a_1\pi_t)= a_1+a_2+\de_{i,j}$ where $\de_{i,j}$ is given in \eqref{eq:201}.

\item 
We next perform a similar construction 
in order to connect $\a_1\pi_t$ to an image of the first vertex $\pi_{n_2}$
of $\rho$. These two vertices have types $o_i$ and $o_j$, as before, and thus we insert
the SAW  $\nu(\a_1\pi_t,\a_2\pi_{n_2}) :=\a_1\g\nu(o_i,v_j)$, where 
$\a_2=\a_1\g v_j\pi_{n_2}^{-1}$ is the unique element of $\sH$
such that $\a_2\pi_{n_2} = \a_1\g v_j$. 
The union of the three SAWs $\pi(2)$, $\nu(\a_1\pi_t,\a_2\pi_{n_2})$, and 
$\rho':=\a_2\rho$ is a SAW denoted $\pi'$ from $\id$ to $\a_2\pi_n$.
\end{numlist}

The resulting SAW  $\pi'$ is a half-space SAW from $v$ with $\len$-length satisfying
$\len(\pi') \in [m+2s_{i,j},m +2s_{i,j}+A)$.
The maximal bridges and reversed bridges that compose $\pi'$ have 
height changes $a_1+a_2+a_3+2\de_{i,j}, a_4,\dots, a_k$. 
Note that $\pi'$ is obtained by (i) re-ordering
the steps of $\pi$, (ii) reversing certain steps, and (iii) adding paths with 
(multiplicative) aggregate $\len$-weight $\th_{i,j}^2$. 
In particular, \eqref{pi'in11}--\eqref{pi'in22} hold.
If $k=2$, the construction ceases after Step 2.

We consider next the multiplicities associated with the map $\pi\mapsto\pi'$. 
The argument in the proof of 
\cite[Sect.\ 7]{GL-loc} may not be used directly since the graph $K_\phi$ 
is not assumed locally finite.
The argument required here is however considerably 
simpler since we are working with Cayley graphs.  

We assume $k \ge 3$ (the case $k=2$ is similar).
In the above construction, we map 
$\pi=(\pi(1),\si,\rho)$ to $ \pi'=(\pi(1),\si',\rho')$ as in Figure \ref{fig:surgery},
where $\si' =\a_1\si$ and $\rho'=\a_2\rho$. Given $\pi'$ and the integers $(a_i)$,
$\pi(\id)$ is given uniquely  as 
the shortest SAW of $\pi'$ starting at $\id$ with span $a_1$.
Having determined $\pi(\id)$, there are no greater than $r+1$ possible choices for
the inserted walk $\nu_1:=\nu(\pi_t,\a_1\pi_{n_2})$, 
where  $r$ is given after \eqref{eq:defd}.
For each such choice of $\nu_1$,  we may identify $\si'$ as the longest sub-SAW 
of $\pi'$ that is a bridge starting at the second endpoint $z$, say, of $\nu_1$ and 
ending at a vertex with height $h(z)+a_2$. 
The second inserted SAW $\nu_2:=\nu(\a_1\pi_t,\a_2\pi_{n_2})$,
is simply $\a_1\nu_1$.
The SAW $\rho'$ is all that remains in $\pi'$. In conclusion, for given
$\pi'$ and $(a_i)$, there are no greater than $r+1$ admissible ways
to express $\pi'$ in the form $(\pi(\id),\si',\rho')$.

Let $(\pi(\id),\si',\rho')$ be such an admissible representation of $\pi'$.
Since $\sH$ acts by left-multiplication, there exists a unique $\a_2\in\sH$
such that $\a_2\pi_n$ equals the final endpoint of $\rho'$. It follows that 
$\rho=\a_2^{-1}\rho'$. By a similar argument, there exists a unique choice for $\si$ that 
corresponds to the given representation.

In conclusion, for given $\pi'$ and $(a_i)$, 
there are no greater than $r+1$ SAWs $\pi$ that give rise to $\pi'$.
By \eqref{pi'in11}--\eqref{pi'in22},
\begin{equation}\label{eq:568}
w(H_m(v)) \le \begin{cases}
\dfrac{r+1}{\thmin^2}
w(B_{m+2s_{i,j}}^v(a_1+a_2+a_3+2\de_{i,j},a_4,\dots,a_k)) 
    &\text{if } k \ge 3,\\
\rule{0pt}{22pt}
\dfrac{r+1}{\thmin}
w(B_{m+s_{i,j}}^v(a_1+a_2+\de_{i,j})) 
    &\text{if } k=2,
\end{cases}
\end{equation}

Since $\len\in\La_\eps(C,\phi)$ where $\eps\in [1,2)$,  by Proposition \ref{prop2} 
the span of any SAW $\pi''$ with $w(\pi'')>0$ is no greater than $C\len(\pi'')^\eps$. 
Write $\sum_a^{(k,T)}$ for the summation
over all finite integer sequences $a_1>\dots>a_k>0$ with 
length $k$ and sum $T$. 

By iteration of \eqref{eq:568}, as in \cite[Lemma 6.1]{GL-loc},
\begin{align*}
w(H_m(v))&\le \sum_{T=1}^{M} \sum_{k=1}^{\sqrt{2T}} \sum_a^{(k,T)} 
w(B_m^v (a_1,\dots,a_k))\\
&\leq \sum_{T=1}^{M} \sum_{k=1}^{\sqrt{2T}} \sum_a^{(k,T)}
  \left[\left(\frac{r+1}{\thmin}\right)^{k-1} 
\, \sum_{t=0}^{(k-1)s} w(\be_{m+t}(v))\right],
\end{align*}
where 
$$
M=C(m+s\sqrt{2M})^\eps \le C(m+Dm^{\eps/2})^\eps,
$$
and $D = D(s)$. 

Assume for the moment that $\be_{\phi,\len} > 1$;
a similar argument is valid when $\be_{\phi,\len}\le 1$. By \eqref{eq:208} and \eqref{eq:207}, 
there exist constants $D=D(\psi,A,a,\thmin,\thmax,r,s)$, which are 
continuous in $\psi$, $A$, $a$, $\thmin$, $\thmax$, such that
\begin{align*}
w(H_m(v)) &\le   
\sum_{T=1}^{M} \sum_{k=1}^{\sqrt{2T}} \sum_a^{(k,T)}  
\left(\frac{r+1}{\thmin}\right)^{k-1}Dks  (\be_{\phi,\len})^{m+ks+A} \\
&\le D\sum_{T=1}^{M} \sum_{k=1}^{\sqrt{2T}} \sum_a^{(k,T)} 
\left(\frac{r+1}{\thmin}  \right)^{\sqrt {2M}}\bigl(s\sqrt{2M}\bigr)(\be_{\phi,\len})^{m+ks+A}\\
&\le D m^{2\eps}e^{D m^{\eps/2}}(\be_{\phi,\len})^{m+Dm^{\eps/2}},
\end{align*}
as required.
\end{proof}

\begin{proof}[Proof of Proposition \ref{swb4}]
We adapt the proof of \cite[Prop.\ 6.6]{GL-loc}.
In preparation, for $u\in\Ga$, there exists $\eta_u\in\Ga$ such that
$e_u:=\langle u\eta_u,u\rangle\in E_\phi$ and $h(u\eta_u)<h(u)$.
Since $\sH$ acts quasi-transitively, we may pick the $\eta_u$ such that
$$
J:=\bigl\lceil\sup\{\len(e_u): u \in \Ga\}\bigr\rceil
$$
satisfies $J<\oo$.

Let $\pi=(\pi_0,\pi_1,\dots, \pi_n)$ be a SAW 
contributing to $\si_m(\id)$, and let $H=\min\{h(\pi_i): 0 \le i \le n\}$
and $I=\max\{i: h(\pi_i) = H\}$.
We add the edge $e_I$
to $\pi$ to obtain two half-space walks, one from $\pi_I\eta_{\pi_I}$ 
to $\pi_0$ and the other from
$\pi_I$ to $\pi_n$, with respective $\len$-lengths $l+\len(e_I)$ and $L-l$ 
for some $L \in [m,m+A)$ and $l\in[0,\len(\pi)]$.

Therefore, $h_a:= \max_v w(H_a(v))$ satisfies
\begin{equation*}
w(\si_m)\leq \sum
h_a h_b,
\end{equation*}
where the sum is over integers $a$, $b$ satisfying $l \in [a-J,a+J+A)$,
$L-l\in [b,b+A)$. By these constraints,
$$
m-2A-J \le a+b\le m+A+J.
$$
The claim follows from Proposition \ref{hsb3} as in \cite{GL-loc}.
\end{proof}

\begin{proof}[Proof of Theorem \ref{thm:group}]
That $\mu_{\phi,\len}=\be_{\phi,\len}$ 
follows immediately from Proposition \ref{swb4},
and we turn to \eqref{eq:sigmalim}. Let $c \ge A$.
Since $w(\si_{m,c}) \ge w(\be_{m,c})$, we have
$$
\liminf_{m\to\oo} w(\si_{m,c})^{1/m}\ge \be_{\phi,\len},
$$ 
and hence $w(\si_{m,c})^{1/m} \to \mu_{\phi,\len}$ as required in part (a) of the theorem.
Part (b) holds by \eqref{eq:cond1} and the definition \eqref{eq:wtprod} of the weight of a SAW.
\end{proof}

\section{Proof of Theorem \ref{thm:cont2}}\label{sec:pfcont}

We begin with a proposition, and then prove Theorem \ref{thm:cont2}. The section ends
with a proof that, as the truncation of a weight function is 
progressively removed, the connective constant converges
to its original value (see \eqref{eq:104} and Proposition \ref{prop:trunc}).

\begin{proposition}\label{lem:lem4}
\mbox{}
\begin{letlist}
\item
Let $(\phi,\len_1),(\nu,\len_2) \in\Phi \circ\La_\eps(C,W)$ for some $\eps\in[1,2)$ and $C,W>0$. If 
\begin{equation}\label{eq:211'}
\de:= \Dsup\bigl((\phi,\len_1), (\nu,\len_2)\bigr) \q\text{satisfies}\q \de\in(0,1),
\end{equation}
then
\begin{equation}\label{eq:212}
\mu_{\nu,\len_2} \le \begin{cases}
B^{\De\log\De}
(\mu_{\phi,\len_1})^{\De} &\text{if } \mu_{\phi,\len_1} > 1,\\
B^{\De\log\De}
(\mu_{\phi,\len_1})^{1/\De} &\text{if } \mu_{\phi,\len_1}\le 1,
\end{cases}
\end{equation}
where
\begin{equation}\label{eq:new949}
\De= \frac{1+\de}{1-\de},\q B=\exp\left(1/\len_{1,\min}\right),\q 
\len_{1,\min}=\min\bigl\{\len_1(\g): \g\in\supp(\phi)\bigr\}.
\end{equation}

\item If $\nu=\phi\in\Phi$, $\len_1,\len_2\in\La$,
and $\de$ is given by \eqref{eq:211'} and satisfies $\de \in (0,1)$, 
then \eqref{eq:212} holds with $B=1$.
\end{letlist}
\end{proposition}

\begin{proof}
(a) By \eqref{eq:211'}, \eqref{eq:dist3}, and the assumption $\de<1$, we have that
\begin{gather}\label{eq:supp}
\supp(\phi)=\supp(\nu),\\
\label{eq:211}
\frac1\De < \frac{\len_1(\g)}{\len_2(\g)},\frac{\phi(\g)}{\nu(\g)} < \De, 
\qq  \g \in \supp(\phi). 
\end{gather}

Let $\pi=(\pi_0,\pi_1,\dots,\pi_n)$ be such that $w_\phi(\pi)>0$. Then
\begin{align}
\frac{w_{\nu}(\pi)}{w_{\phi}(\pi)} &= \prod_{e \in \pi} \frac {\nu(e)}{\phi(e)}
\le \prod_{e\in\pi} \De  \qq\text{by \eqref{eq:211}}\label{eq:218}\\
&= \De ^n \le \exp\left( \frac{\len_1(\pi)\log\De}{\len_{1,\min}}\right).
\nonumber
\end{align}
By  \eqref{eq:211} and \eqref{eq:218},
\begin{equation}\label{eq:214}
w_\nu(\si_m) \le \exp\left( \frac{R\log\De}{\len_{1,\min}}\right) \sum_{[L,R)} w_\phi(\pi)
\end{equation}
where the summation is over SAWs $\pi$ from $\id$ with $L\le \len_1(\pi)<R$, and 
$$
L=\frac m\De, \qq 
R=(m+A)\De.
$$
We take the $m$th root of \eqref{eq:214} and let $m\to\oo$ to 
obtain \eqref{eq:212} by \eqref{eq:sigmalim}.

(b) When $\phi=\nu\in\Phi$, the left side of \eqref{eq:218} equals $1$, 
and the exponential term in \eqref{eq:214} is replaced by $1$.
\end{proof}

\begin{proof}[Proof of Theorem \ref{thm:cont2}]

Let $\Ga$ be virtually $(\sH,F)$-indicable. Let $C,W>0$, $\eps\in[1,2)$,
and  $(\phi,\len_1)\in \Phi\circ\La_\eps(C,W)$.
We shall assume that $\mu_{\phi,\len_1} \ge 1$ (implying that $W \ge 1$); 
a similar proof is valid otherwise.
We shall show continuity at the point $(\phi,\len_1)$. 

Let $\rho>0$, and write
\begin{equation}\label{eq:new941}
D(\de):=W\Bigl| B^{\De\log\De}W^{\De-1}-1\Bigr|
 +W\Bigl| B^{\De^2\log\De}W^{\De-1}-1\Bigr|,
\end{equation}
where $B$ and $\De=\De(\de)$ are given in \eqref{eq:new949}.
Pick $\de_0 \in (0,1)$ such that
\begin{equation}\label{eq:new942}
D(\de)<\rho\qq\text{whenever} \qq \de\in(0,\de_0).
\end{equation}

Let $(\nu,\len_2) \in \Phi \circ \La_\eps(C,W)$,
write $\de := \Dsup\bigl((\phi,\len_1), (\nu,\len_2)\bigr)$, and suppose $\de<\de_0$.
We shall prove that 
\begin{equation}
\bigl|\mu_{\phi,\len_1} -\mu_{\nu,\len_2}\bigr| < \rho.\label{mm3}
\end{equation}

Note by \eqref{eq:211} that, since $\de<\de_0<1$,
\begin{equation}\label{eq:new779}
\len_{2,\min} \ge \frac{ \len_{1,\min}}\De.
\end{equation}

Assume that $\mu_{\nu,\len_2}\ge 1$; a similar proof is valid otherwise.
By \eqref{eq:new779} and Proposition \ref{lem:lem4}(a),
\begin{align}\label{eq:new908}
\mu_{\phi,\len_1}-
B^{\De\log\De}(\mu_{\phi,\len_1})^{\De}
&\leq \mu_{\phi,\len_1}-\mu_{\nu,\len_2}\\%\nonumber\\
&\leq B^{\De^2\log\De} (\mu_{\nu,\len_2})^{\De}-\mu_{\nu,\len_2}.
\nonumber
\end{align}
Therefore,
\begin{equation*}
\bigl| \mu_{\phi,\len_1}-\mu_{\nu,\len_2}\bigr| \le D(\de) < \rho,
\end{equation*}
by \eqref{eq:new941} and \eqref{eq:new942}. Inequality \eqref{mm3} is proved.
\end{proof}

Let $\phi\in \Phi$. For $\eta>0$, let $\phie:\Ga\to[0,\oo)$ be the truncated weight function
\begin{equation}\label{eq:104}
\phie(\g) = \begin{cases} \phi(\g) &\text{if } \phi(\g)\ge\eta,\\
0 &\text{otherwise},
\end{cases}
\end{equation}
and let $M=M(\phi)= \sup\{\eta>0: 
\phie \text{ spans } \Ga\}$. By Lemma \ref{l41}, $M>0$.

\begin{definition}\label{def:cts}
Let $\phi\in\Phi$ and $\len\in\La$. We say that the pair $(\phi,\len)$ 
is \emph{continuous at $0$} if $\len(\g)\to\oo$ as $\phi(\g)\to 0$, which is to say that
\begin{equation}\label{eq:ass3}
\text{$\forall K>0,\, \exists \eta>0$ such that: $\forall \g$ satisfying $\phi(\g)\in(0,\eta)$, we have $\len(\g)>K$.}
\end{equation} 
The set of such pairs $(\phi,\len)$  is denoted $\sC$.
%
%\grg{}
%For $N\subseteq\Phi$, the pair $(N,\len)$ is said to be \emph{uniformly continuous at $0$} if, 
%for $K>0$,    the value $\eta$ in \eqref{eq:ass3} may be chosen
%uniformly for $\phi\in N$.
\end{definition}

\begin{example}
If $\len\in\La$ is such that $\len(\g)=1/\phi(\g)$ on $\supp(\phi)$,
then $(\phi,\len)$ is continuous at $0$.
Recall Example \ref{ex:weighted}.
\end{example}

\begin{proposition}\label{prop:trunc}
Let $\Ga$ be finitely generated and virtually $(\sH,F)$-indicable, and let
$(\phi,\len)\in\Phi\circ\La_\eps(C)$ for some $C>0$ and $\eps\in [1,2)$. 
Assume in addition that \eqref{eq:ass3} holds, in that $(\phi,\len)\in\sC$. Then 
$$
\mu_{\phie,\len} \to \mu_{\phi,\len} \qq\text{as } \eta \to 0.
$$
\end{proposition}

\begin{proof}

Since $\len\in\La_\eps(C,\phi)$ and $\phi^\eta\le\phi$, we have by
\eqref{eq:cond1} that $\len\in\La_\eps(C,\phi^\eta)$ for $\eta<M$.
Therefore, Theorem \ref{thm:group} may be applied to the pairs $(\phi^\eta,\len)$.
%The inequalities 
%\begin{equation}\label{eq:ineqs}
%\mu_{\phie,\len}\le \mu_{\phi,\len}, \qq \be_{\phie,\len}\le \be_{\phi,\len},
%\end{equation}
%hold since $\phie\le\phi$ and $\mu_{\phie,\len}$, $\be_{\phie,\len}$ 
%are  non-decreasing in $\eta$. 

Let $\zeta\in(0,M)$, so that $\phi^\zeta$ spans $\Ga$. 
Let $K_{\phi^\zeta}$ be the locally finite Cayley graph of
$\Ga$ on the edges $e$ for which $\phi^\zeta(e)>0$
(see Lemma \ref{l41}), with corresponding \sghf\ $(h,\sH)$, and let 
$\psi$, $A$, $a$ be as in \eqref{eq:102} with $\phi$ replaced by $\phi^\zeta$.
Working on the graph $K_{\phi^\zeta}$ with the weight function
$\phi^\zeta$, we construct the paths $\nu_{i,j}$ as before \eqref{eq:phimin-}, and we shall stay with these 
particular paths in the rest of the proof. Since $\phi^\eta$ is constant on these
paths for $\eta \le \zeta$,
the values of $\psi$, $A$, $a$, $\thmin$, $\thmax$, $r$, $s$ are unchanged for $\eta\le \zeta$.

Assume $\be_{\phi,\len} >1$; a similar proof holds if 
$\be_{\phi,\len} \le 1$. Take $c=A $ in Theorem \ref{thm:group}, and
let $m\in\NN$; later we shall allow $m\to\oo$. 
By \eqref{eq:ass3}, we may choose $\rho=\rho_m\in(0,\zeta)$, such that
\begin{equation}\label{eq:new777}
\text{for any $\gamma$ satisfying $\phi(\gamma)\in (0,\rho)$, we have $\len(\gamma)>m+A$.}
\end{equation}

We shall write $w_\phi$ since we shall work with more than one weight function.
By \eqref{eq:104}--\eqref{eq:new777},
\begin{align}\label{eq:217}
w_\phi(\si_{m,A}) =w_{\phi^\rho}(\si_{m,A}).
\end{align}
Since $\phi^\rho \le \phi$, we have by \eqref{eq:cond1}  
that $\len\in\La_\eps(C,\phi^\rho)$.
By Proposition \ref{swb4} applied to $\phi^\rho$, 
\begin{equation}\label{eq:2175}
w_{\phi^\rho}(\si_{m,A})\leq g_m (\be_{\phi^\rho,\len})^{m+Dm^{\eps/2}},
\end{equation}
where $g_m=Dm^{4\eps} e^{D m^{\eps/2}}$
with the constants $D=D(\psi,A,a,\thmin,\thmax,r,s)$  
given as in the proposition. Since $\be_{\phi^\rho,\len} \le \mu_{\phi^\rho,\len}$
and (by Theorem \ref{thm:group}(b)) $\mu_{\nu,\len}$ is non-decreasing in the weight function $\nu$, 
we have by \eqref{eq:217}--\eqref{eq:2175} that
\begin{equation}\label{eq:new917}
w_\phi(\si_{m,A}) \le g_m(\mu_{\phi^\rho,\len})^{m+Dm^{\eps/2}}
\le g_m \left(\lim_{\eta\to 0}\mu_{\phie,\len}\right)^{m+Dm^{\eps/2}}.
\end{equation}
Take $m$th roots and let $m\to\oo$, to obtain by \eqref{eq:sigmalim} that
$$
\mu_{\phi,\len} \le \lim_{\eta\to 0} \mu_{\phie,\len},
$$
and the result follows. 
\end{proof}

\section{Proof of Theorem \ref{thm1}}
\label{sec:pf2}

We omit this proof since it lies close to those of Theorem \ref{thm:group}
and \cite[Thm 4.3]{GL-loc}. Here is a summary of the differences. 
Theorem \ref{thm1} differs from
\cite[Thm 4.3]{GL-loc} in that edges are weighted, and this difference is handled
in very much the same manner as in the proof of Theorem \ref{thm:group}.
On the other hand, Theorem \ref{thm1}  differs from Theorem \ref{thm:group} in that
the underlying graph need not be a Cayley graph. This is handled as in
the proof of  \cite[Thm 4.3]{GL-loc}, where the unimodularity of
the automorphism groups in question is utilized. See also Remark \ref{lind1}.

\section*{Acknowledgements}
We thank Agelos Georgakopoulos for enquiring about weighted SAWs on groups.
GRG was supported in part by the EPSRC under grant EP/I03372X/1.
ZL was supported in part by the NSF under grant DMS-1608896.
Some of the work was done during a visit by GRG to Keio University, Tokyo.
The authors are grateful to two referees for their useful comments.

\providecommand{\bysame}{\leavevmode\hbox to3em{\hrulefill}\thinspace}
\providecommand{\MR}{\relax\ifhmode\unskip\space\fi MR }
% \MRhref is called by the amsart/book/proc definition of \MR.
\providecommand{\MRhref}[2]{%
  \href{http://www.ams.org/mathscinet-getitem?mr=#1}{#2}
}
\providecommand{\href}[2]{#2}

%\bibliography{sawfinal-23}
%\bibliographystyle{amsplain}
\end{document}